\documentclass[12pt,a4paper, reqno]{amsart}
\usepackage[a4paper]{geometry}
\usepackage{amsthm}
\usepackage{amssymb}
\usepackage{amsmath}
\usepackage{mathtools}
\usepackage[utf8]{inputenc}
\usepackage{comment}

\usepackage[T1]{fontenc}

\usepackage{hyperref}

\usepackage{wasysym}
\usepackage{color}
\usepackage{enumitem}

\usepackage{tikz}
\usetikzlibrary{positioning}
\usetikzlibrary{decorations.pathreplacing}
\usetikzlibrary{patterns}
\pgfdeclarepatternformonly[\LineSpace]{my north east lines}{\pgfqpoint{-1pt}{-1pt}}{\pgfqpoint{\LineSpace}{\LineSpace}}{\pgfqpoint{\LineSpace}{\LineSpace}}%
{
    \pgfsetlinewidth{0.4pt}
    \pgfpathmoveto{\pgfqpoint{0pt}{0pt}}
    \pgfpathlineto{\pgfqpoint{\LineSpace + 0.1pt}{\LineSpace + 0.1pt}}
    \pgfusepath{stroke}
}

\newdimen\LineSpace
\tikzset{
    line space/.code={\LineSpace=#1},
    line space=3pt
}

\newtheorem{thm}{Theorem}[section]
\newtheorem{cor}[thm]{Corollary}
\newtheorem{lem}[thm]{Lemma}
\newtheorem{prop}[thm]{Proposition}

\theoremstyle{definition}

\newtheorem{eks}[thm]{\sc Example}
\theoremstyle{remark}
\newtheorem{rem}[thm]{Remark}

\numberwithin{equation}{section}

\DeclareMathOperator{\Lip}{Lip}

\DeclareMathOperator{\dent}{dent}

\DeclareMathOperator{\clconv}{\overline{conv}}
\DeclareMathOperator{\conv}{conv}
\DeclareMathOperator{\supp}{supp}

\DeclareMathOperator{\N}{\mathbb{N}}
\newcommand{\seg}[3]{[#1,#2]_{#3}}

\title[Daugavet- and Delta-points in spaces of Lipschitz functions]{Daugavet- and Delta-points in spaces of Lipschitz functions}
\author{Triinu Veeorg}
\address{Institute of Mathematics and Statistics, University of Tartu, Narva~mnt 18, 51009, Tartu, Estonia}
\email{triinu.veeorg@ut.ee}

\begin{document}
\begin{abstract}
    A norm one element $x$ of a Banach space is a Daugavet-point (respectively,~a $\Delta$-point) if every slice of the unit ball (respectively,~every slice of the unit ball containing $x$) contains an element that is almost at distance 2 from $x$. We prove the equivalence of Daugavet- and $\Delta$-points in spaces of Lipschitz functions over proper metric spaces and provide two characterizations for them. Furthermore, we show that in some spaces of Lipschitz functions, there exist $\Delta$-points that are not Daugavet-points. Lastly, we prove that every space of Lipschitz functions over an infinite metric space contains a $\Delta$-point but might not contain any Daugavet-points.
\end{abstract}

\maketitle

\section{Introduction}
Let $X$ be a Banach space and $x\in S_X$. According to \cite{AHLP} we say that
\begin{enumerate}
    \item $x$ in $S_X$ is a \emph{Daugavet-point} if $B_X = \clconv \Delta_\varepsilon(x)$ for every $\varepsilon>0$,\smallskip
    \item $x$ in $S_X$ is a \emph{$\Delta$-point} if $x\in \clconv\Delta_\varepsilon(x)$ for every $\varepsilon>0$, 
\end{enumerate}
    where
$$\Delta_\varepsilon(x) = \big\{y\in B_X\colon \|x-y\|\geq 2-\varepsilon\big\}.$$
These concepts were introduced as local versions of 
the Daugavet property and the diametral local diameter two property, respectively (see \cite{AHLP}, \cite{DW01}).  The properties of these points have been studied in various papers (see, e.g., \cite{ALM}, \cite{DJR}, \cite{HPV}, \cite{JRZ}, \cite{Vee}).

Jung and Rueda Zoca started the study of Daugavet-points in spaces of Lipschitz functions in \cite{JRZ} and so far it has been shown that every local function is a Daugavet-point (see {\cite[Proposition 3.4, Theorem~3.6]{JRZ}}, {\cite[Theorem~1.4]{HOP}}). 
The purpose of this paper is to extend our knowledge of these points in spaces of Lipschitz functions.  Throughout the paper, $M$ is a metric space with metric $d$ and a fixed point 0. We denote by $\Lip_0(M)$ the Banach space of all Lipschitz functions $f\colon M\rightarrow\mathbb{R}$ with $f(0)=0$ equipped with the obvious linear structure and the norm
$$\|f\|:=\sup\Big\{\frac{|f(x)-f(y)|}{d(x,y)}\colon x, y\in M, x\neq y\Big\}.$$ 
Let $\delta\colon M\rightarrow \Lip_0(M)^*$ be the canonical isometric embedding of $M$ into $\Lip_0(M)^*$, which is given by $x\mapsto \delta_x$ where $\delta_x(f)=f(x)$. 
The norm closed linear span of $\delta(M)$ in $\Lip_0(M)^*$ is called the \emph{Lipschitz-free space over $M$} and is denoted by $\mathcal{F}(M)$ (see \cite{GOD} and \cite{Weaver} for the background).  An element in $\mathcal{F}(M)$ of the form 
$$m_{xy}:=\frac{\delta_x-\delta_y}{d(x,y)}$$
for $x, y\in M$ with $x\neq y$ is called a \emph{molecule}. We denote the set of all molecules by $\mathcal{M}(M)$.
Clearly, $\mathcal{M}(M)\subseteq S_{\mathcal{F}(M)}$, and it is well known that
\[\clconv\big(\mathcal{M}(M)\big)=B_{\mathcal{F}(M)}\quad\text{ and }\quad
\mathcal{F}(M)^*= \Lip_0(M).\]

In Section 2, we introduce some additional definitions and provide two new results that are later used in Sections 3 and 4. We introduce a stronger condition that all $\Delta$-points must satisfy. This tool also provides an easy proof that quasi denting points can not be $\Delta$-points. Furthermore, we show that if $f\in S_{\Lip_0(M)}$ is not local, then every slice of $B_{\mathcal{F}(M)}$ whose defining functional is $f$ contains a denting point of $B_{\mathcal{F}(M)}$, thus generalizing {\cite[Lemma~3.13]{CGMR}}.

In Section 3, we give a complete characterization for $\Delta$-points in Lipschitz-free spaces over proper metric spaces, thus giving a positive answer for Problem 2 in \cite{Vee} for the case when $M$ is a proper metric space. The same characterization was first provided for molecules (see {\cite[Theorem~4.7]{JRZ}}) and later generalized for convex combinations of molecules (see {\cite[Theorem~4.4]{Vee}}). 

Section 4 is dedicated to studying Daugavet- and $\Delta$-points in spaces of Lipschitz functions over proper metric spaces. We prove that Daugavet- and $\Delta$-points are equivalent in these spaces and are also equivalent to their $w^*$-versions. We finish the section by adapting that result to a different class of metric spaces.

In Section 5, we prove that every space of Lipschitz functions over an infinite metric space has a $\Delta$-point. We also show that spaces of Lipschitz functions over infinite metric spaces that satisfy certain conditions (for example, are unbounded or not uniformly discrete)  contain Daugavet-points. Furthermore, we provide an example of infinite-dimensional space of Lipschitz functions that does not have any Daugavet-points nor even $w^*$-Daugavet-points. We finish the paper by showing that in spaces of Lipschitz functions, Daugavet- and $\Delta$-points are not equivalent, nor are they equivalent to their $w^*$-versions.

We consider only real Banach spaces and use standard notation. For a Banach space $X$ we will denote the closed unit ball by $B_X$, the unit sphere by $S_X$ and the dual space by $X^*$.  By $B(x,r)$ we denote the closed ball of radius $r$ with the center $x$.

\section{Definitions and preliminary results}
Both Daugavet- and $\Delta$-points have equivalent definitions using slices. By a \emph{slice of the unit ball $B_X$} we mean the set
\[
S(x^*, \alpha)=\{y\in B_X \colon x^*(y)>1-\alpha\}
\]
where $x^*\in S_{X^*}$ and $\alpha>0$. By a \emph{$w^*$-slice of the unit ball $B_{X^*}$} we mean a slice of $B_{X^*}$ whose defining functional belongs to $S_X$.
In our results we use both the original definitions of Daugavet- and $\Delta$-points as well as the following equivalent definitions by slices (see {\cite[Lemmas 2.1, 2.2]{AHLP}}):
\begin{enumerate}
    \item $x\in S_X$ is a Daugavet-point if and only if for every slice $S$ of $B_X$ and for every $\varepsilon>0$ there 
exists $y\in S$ such that $\|x-y\|\ge 2-\varepsilon$;
    \item $x\in S_X$ is a $\Delta$-point if and only if for every slice $S$ of $B_X$ with $x\in S$ and for every $\varepsilon>0$ there 
exists $y\in S$ such that $\|x-y\|\ge 2-\varepsilon$.
\end{enumerate}
As our main results are in spaces of Lipschitz functions, which are always dual spaces, we will also examine the natural $w^*$-versions of these points: 
\begin{enumerate}
    \item $x^*\in S_{X^*}$ is a \emph{$w^*$-Daugavet-point} if for every $w^*$-slice $S$ of $B_{X^*}$ and for every $\varepsilon>0$ there 
exists $y^*\in S$ such that $\|x^*-y^*\|\ge 2-\varepsilon$;
    \item $x^*\in S_{X^*}$ is a \emph{$w^*$-$\Delta$-point} if for every $w^*$-slice $S$ of $B_{X^*}$ with $x^*\in S$ and for every $\varepsilon>0$ there 
exists $y^*\in S$ such that $\|x^*-y^*\|\ge 2-\varepsilon$.
\end{enumerate}

Now we are ready to introduce a criterion for $\Delta$-points that allows us to show the equivalence of Daugavet- and $\Delta$-points in spaces of Lipschitz functions over proper metric spaces and provides a nice characterization for them. 

\begin{prop}\label{general}
Let $x_0\in S_{X}$ be a $\Delta$-point. Then for every slice $S(x_0^*,\alpha)$ with $x_0\in S(x_0^*,\alpha)$ there exist sequences $(x_i)$ in $S(x_0^*,\alpha)$ and $(x_i^*)$ in $S_{X^*}$ such that $x_i^*\in S(x_0,1/i)$ for all $i\in \N$, and
\[\|x_i-x_j\|\ge 2-\alpha \quad\text{ and }\quad\|x_i^*-x_j^*\|\ge 2-\alpha\]
for all $i,j\in \N\cup\{0\}$ with $i\neq j$.
\end{prop}

\begin{proof}
Fix a slice $S(x_0^*,\alpha)$ with $x_0\in S(x_0^*,\alpha)$. Let $\gamma>0$ be such that $x_0\in S(x_0^*,\alpha-2\gamma)$ and $\gamma<1$. We will construct the sequences $(x_i)$ and $(x_i^*)$ recursively.
Choose $x_1\in S(x_0^*,\alpha-\gamma)$ such that $\|x_0-x_1\|>2-\gamma/2$, and then choose $x_1^*\in S_{X^*}$ such that $x_1^*(x_0)>1-\gamma/2$ and $x_1^*(-x_1)>1-\gamma/2$. Then
\[\|x_0-x_1\|>2-\frac{\gamma}{2}> 2-\alpha\]
and
\[\|x_0^*-x_1^*\|\ge x_0^*(x_1)-x_1^*(x_1)>1-\alpha+\gamma+1-\frac{\gamma}{2}>2-\alpha.\]

Assume that we have found $x_1,\ldots,x_{n-1}\in S_{X}$ and $x_1^*,\ldots,x_{n-1}^*\in S_{X^*}$ such that 
\[x_i\in S(x_0^*,\alpha-\gamma)\quad\text{ and }\quad x_0,-x_i\in S(x_i^*,\gamma/2^i)\]
for every $i\in \{1,\ldots,n-1\}$, and 
\[\|x_i-x_j\|\ge 2-\alpha \quad\text{ and }\quad\|x_i^*-x_j^*\|\ge 2-\alpha\]
for every $i,j\in \{0,\ldots,n-1\}$ with $i\neq j$.
Let 
\[y^*=\sum_{i=0}^{n-1} x_i^*.\]
Then 
\[y^*(x_0)=\sum_{i=0}^{n-1} x_i^*(x_0)>1-\alpha+2\gamma+\sum_{i=1}^{n-1}(1-\gamma/2^i)>n-\alpha+\gamma.\]
Since $x_0$ is a $\Delta$-point, there exist $x_n\in S_X$ such that
\[y^*(x_n)>n-\alpha+\gamma\]
and $\|x_0-x_n\|>2-\gamma/2^n$. Choose $x_n^*\in S_{X^*}$ such that $x_n^*(x_0-x_n)>2-\gamma/2^n$. Then $x_0,-x_n\in S(x_n^*,\gamma/2^n)$ and $x_n\in S(x_i^*,\alpha-\gamma)$ for every $i\in \{0,\ldots,n-1\}$.
Furthermore,
\[\|x_n-x_i\|\ge x_i^*(x_n)-x_i^*(x_i)>1-\alpha+\gamma+1-\frac{\gamma}{2^i}>2-\alpha\]
for every $i\in \{1,\ldots,n-1\}$, and
\[\|x_n^*-x_i^*\|\ge x_i^*(x_n)-x_n^*(x_n)>1-\alpha+\gamma+1-\frac{\gamma}{2^n}>2-\alpha\]
for every $i\in \{0,\ldots,n-1\}$.
\end{proof}

Proposition \ref{general} allows us to show that Kuratowski measure of non-compactness is 2 for each slice containing a $\Delta$-point, thus providing us stronger versions of {\cite[Theorems~3.5, 4.2]{ALMP}}. Recall that for a bounded subset $A\subseteq X$, the \emph{Kuratowski measure of non-compactness} $\alpha(A)$ is the infimum of such $\varepsilon>0$ that $A$ can be covered with a finite number of sets with diameter less than $\varepsilon$.
\begin{cor}[cf. {\cite[Theorems~3.5, 4.2]{ALMP}}]\label{quasi_dent}
Let $x\in S_X$ be a $\Delta$-point. Then $\alpha(S)=2$ for every slice $S$ of $B_X$ with $x\in S$. Furthermore, $\alpha\big(S(x,\delta)\big)=2$ for every $\delta>0$.
\end{cor}

\begin{proof}
By Proposition \ref{general} any slice $S(x^*,\delta)$ of $B_X$ with $x\in S(x^*,\delta)$ contains countable number of points such that any two points are at least at distance $2-\delta$ from each other. Therefore covering $S(x^*,\delta)$ with a finite number of sets means that at least one of these sets must contain two points that are at least at distance $2-\delta$ from each other. Now recall that by {\cite[Lemma~2.1]{IK}} for every $\beta\in(0,\delta)$ there exist $y^*\in S_X$ such that $x\in S(y^*,\beta)$ and $S(y^*,\beta)\subseteq S(x^*,\delta)$ and thus 
\[\alpha(S(x^*,\delta))\ge \alpha(S(y^*,\beta))\ge 2-\beta.\]
Therefore $\alpha(S(x^*,\delta))=2$.
The equality $\alpha(S(x,\delta))=2$ for every $\delta>0$ can be shown similarly.
\end{proof}

In \cite{ALMP}, these results were used to show that AUS and reflexive AUC spaces do not contain $\Delta$-points, and at the end of Section 4, it was noted that it is not clear if quasi-denting points can be $\Delta$- or Daugavet-points. Recall that $x\in S_X$ is a quasi-denting point of $B_X$ if for every $\varepsilon>0$ there exist a slice $S$ of $B_X$ with $x\in S$ such that $\alpha(S)<\varepsilon$. From Corollary \ref{quasi_dent} it is clear that a quasi-denting point can not be a $\Delta$-point. 

Similarly to Proposition \ref{general} and Corollary \ref{quasi_dent} we can prove the two following statements for $w^*$-$\Delta$-points.
\begin{prop}\label{general_weak}
Let $x_0^*\in S_{X^*}$ be a $w^*$-$\Delta$-point. Then for every $w^*$-slice $S(x_0,\alpha)$ with $x_0^*\in S(x_0,\alpha)$ there exist sequences $(x_i^*)$ in $S(x_0,\alpha)$ and $(x_i)$ in $S_X$ such that $x_i\in S(x_0^*,1/i)$ for every $i\in\N$, and 
\[\|x_i^*-x_j^*\|\ge 2-\alpha \quad\text{ and }\quad\|x_i-x_j\|\ge 2-\alpha\]
for all $i,j\in \N\cup\{0\}$ with $i\neq j$.
\end{prop}

\begin{cor}
Let $x^*\in S_{X^*}$ be a $w^*$-$\Delta$-point. Then $\alpha(S)=2$ for every $w^*$-slice $S$ of $B_{X^*}$ with $x^*\in S$. Furthermore, $\alpha\big(S(x^*,\delta)\big)=2$ for every $\delta>0$.
\end{cor}

\begin{rem}\label{remark_conv}
From {\cite[Lemma~2.2]{JRZ}} we see that if $A\subseteq B_X$ is such that every subslice of $S(x_0^*,\alpha)$ contains an element from the set $A$, then in Proposition \ref{general} we may choose the elements of the sequence $(x_i)$ from the set $A$. Furthermore, if $A^*\subseteq B_{X^*}$ is such that for some $\varepsilon>0$ every subslice of $S(x_0,\varepsilon)$ contains an element from the set $A^*$, then in Proposition \ref{general} we may choose the elements of the sequence $(x_i^*)$ from the set $A^*$. Similar statements hold for Proposition \ref{general_weak}. In particular, this is useful for examining Lipschitz-free spaces and their duals, since $\clconv(\mathcal{M}(M))=B_{\mathcal{F}(M)}$, i.e., each slice of $B_{\mathcal{F}(M)}$ contains a molecule. 
\end{rem}

As a last part of this section we will adjust the proofs of {\cite[Lemma~2.4]{Vee}} and {\cite[Theorem~2.1]{Vee}} to show that if $f\in S_{\Lip_0(M)}$ is not local, then for every $\alpha>0$ the slice $S(f,\alpha)$ contains a denting point of $B_{\mathcal{F}(M)}$. This has previously been proved for the case of compact metric space (see {\cite[Lemma~3.13]{CGMR}}). An element $f\in \Lip_0(M)$ is called \emph{local} if for every $\varepsilon>0$ there exist $u,v\in M$ with $u\neq v$ such that $f(m_{uv})>\|f\|-\varepsilon$ and $d(u,v)<\varepsilon$. Recall that $x\in B_X$ is a \emph{denting point of $B_X$} if there exist  slices of $B_X$ with arbitrarily small diameter that contain $x$. We denote the set of all denting points of $B_X$ by $\dent (B_X)$.
For every $u,v\in M$ and $\delta>0$ let
\[\seg{u}{v}{\delta}:=\big\{p\in M \colon d(u,p)+d(v,p)<d(u,v)+\delta\big\}.\]
\begin{lem}[cf. {\cite[Lemma~2.4]{Vee}}]\label{lemma_dent}
Assume that $M$ is complete, and let $u, v\in M$ and $r,s,\delta>0$ with $r+s<d(u,v)$ be such that
$$\seg{u}{v}{\delta}\subseteq B(u,r)\cup B(v,s).$$
Then there exist $x\in B(u,r)$ and $y\in B(v,s)$ such that $m_{xy}$ is a denting point of $B_{\mathcal{F}(M)}$ and 
\[d(u,x)+d(v,y)+d(x,y)<d(u,v)+\delta.\]
\end{lem}

\begin{proof}
Construct four sequences $(x_n)$, $(y_n)$, $(\delta_n)$, and $(\varepsilon_n)$ as was done in {\cite[proof of Lemma~2.4]{Vee}}. Then there exist $x,y\in M$ such that $x_n\rightarrow x$ and $y_n\rightarrow y$. Furthermore,
\[d(x_1,x)+d(y_1,y)+d(x,y)\le d(x_1,y_1)+5\varepsilon_{2}\]
and 
\[d(u,x_1)+d(v,y_1)+d(x_1,y_1)<d(u,v)+\delta.\]
Then 
\begin{align*}
    d(u,x)+d(v,y)+d(x,y)&\le d(u,x_1)+d(v,y_1)+ d(x_1,x)+d(y_1,y)+d(x,y)\\
    &<d(u,v)-d(x_1,y_1)+\delta+ d(x_1,y_1)+5\varepsilon_{2}\\
    &<d(u,v)+2\delta.
\end{align*}
Note that $\seg{u}{v}{\delta/2}\subseteq B(u,r)\cup B(v,s)$, and therefore we can start with $\delta/2$ instead of $\delta$, in which case we get
\[d(u,x)+d(v,y)+d(x,y)<d(u,v)+\delta.\]
\end{proof}

\begin{prop}\label{local_dent}
    Let $f\in S_{\Lip_0(M)}$, and assume that $f$ is not local. Then for every $\alpha>0$ the slice $S(f,\alpha)$ contains a denting point of $B_{\mathcal{F}(M)}$.
\end{prop}

\begin{proof}
We may assume that $M$ is complete, since for its completion $M'$ we have $\mathcal{F}(M)=\mathcal{F}(M')$.

Since $f$ is not local, there exists $\varepsilon>0$ such that $d(u,v)>\varepsilon$ for every $m_{uv}\in S(f,\varepsilon)$.
Fix $\alpha>0$. We may assume that $\alpha<\varepsilon$.

Let $m_{u_0v_0}\in S(f,\alpha/2)$. Let $n\in \mathbb{N}$ and $\delta>0$ be such that 
$$\Big(\frac{3}{4}+\delta\Big)^nd(u_0,v_0)<\varepsilon$$
and  $f(u_0)-f(v_0)>(1-\alpha/2+\delta)(1+\delta)^{n}d(u_0,v_0)$.

If
\[\seg{u_0}{v_0}{\delta d(u_0,v_0)}\subseteq  B\big(u_0,d(u_0,v_0)/4\big)\cup B\big(v_0,d(u_0,v_0)/4\big),\]
then by Lemma \ref{lemma_dent} there exist $x\in B\big(u_0,d(u_0,v_0)/4\big)$ and $y\in B\big(v_0,d(u_0,v_0)/4\big)$ such that $m_{xy}$ is a denting point of $B_{\mathcal{F}(M)}$ and 
\[d(u_0,x)+d(v_0,y)+d(x,y)<d(u_0,v_0)+\delta d(u_0,v_0).\]
Then $d(x,y)\ge d(u_0,v_0)/2$ and therefore
\begin{align*}
    f(x)-f(y)&=f(x)-f(u_0)+f(v_0)-f(y)+f(u_0)-f(v_0)\\
    &> -d(x,u_0)-d(v_0,y)+(1-\alpha/2+\delta)d(u_0,v_0)\\
    &>d(x,y)-\alpha/2 d(u_0,v_0)\\
    &\ge (1-\alpha)d(x,y).
\end{align*}
Now we have found a denting point $m_{xy}$ of $B_{\mathcal{F}(M)}$ in the slice $S(f,\alpha)$.

Otherwise there exists 
$$p\in \seg{u_0}{v_0}{\delta d(u_0,v_0)}\setminus\Big( B\big(u_0,d(u_0,v_0)/4\big)\cup B\big(v_0,d(u_0,v_0)/4\big)\Big).$$
Therefore
\begin{align*}
    f(u_0)-f(p)+f(p)-f(v_0)&>(1-\alpha/2+\delta)(1+\delta)^{n}d(u_0,v_0)\\
    &>(1-\alpha/2+\delta)(1+\delta)^{n-1}\big(d(u_0,p)+d(v_0,p)\big).
\end{align*}
Then either $$f(u_0)-f(p)>(1-\alpha/2+\delta)(1+\delta)^{n-1}d(u_0,p)$$
or 
$$f(p)-f(v_0)>(1-\alpha/2+\delta)(1+\delta)^{n-1}d(v_0,p).$$
Additionally we have
$$d(u_0,p)<(1+\delta)d(u_0,v_0)-d(v_0,p)< \Big(\frac{3}{4}+\delta\Big)d(u_0,v_0).$$
Analogously  $d(v_0,p)< \big(3/4+\delta\big)d(u_0,v_0).$
Therefore there exist $u_1, v_1\in M$ such that
$$f(u_1)-f(v_1)>(1-\alpha/2+\delta)(1+\delta)^{n-1}d(u_1,v_1)$$
and $d(u_1,v_1)<(3/4+\delta)d(u_0,v_0)$. 

Now we will repeat this step as many times as needed, but no more than $n$ times. Assume for  $k\in \{1,\ldots,n-1\}$ that
$$f(u_k)-f(v_k)>(1-\alpha/2+\delta)(1+\delta)^{n-k}d(u_k,v_k)$$
and $d(u_k,v_k)<(3/4+\delta)^kd(u_0,v_0)$. 

If
\[\seg{u_k}{v_k}{\delta d(u_k,v_k)}\subseteq  B\big(u_k,d(u_k,v_k)/4\big)\cup B\big(v_k,d(u_k,v_k)/4\big),\]
then we can find a denting point $m_{xy}$ of $B_{\mathcal{F}(M)}$ in the slice $S(f,\alpha)$ as we did before.

Otherwise, we can find $u_{k+1}, v_{k+1}\in M$ such that
$$f(u_{k+1})-f(v_{k+1})>(1-\alpha/2+\delta)(1+\delta)^{n-k-1}d(u_{k+1},v_{k+1})$$
and $d(u_{k+1},v_{k+1})<(3/4+\delta)^{k+1}d(u_0,v_0)$.

By the $n$-th step we must have found a denting point $m_{xy}$ of $B_{\mathcal{F}(M)}$ in the slice $S(f,\alpha)$, because otherwise we would have
$f(u_n)-f(v_n)>(1-\alpha)d(u_n,v_n)$ and 
$$d(u_n,v_n)<\Big(\frac{3}{4}+\delta\Big)^n d(u_0,v_0)<\varepsilon,$$
which is a contradiction. Therefore the slice $S(f,\alpha)$ contains a denting point of $B_{\mathcal{F}(M)}$.
\end{proof}

\section{Delta-points in Lipschitz-free spaces over proper metric spaces}

The study of Daugavet- and $\Delta$-points in Lipschitz-free spaces was started in \cite{JRZ} and subsequently continued in \cite{Vee}. First it was proved that a molecule  $m_{xy}$ is a $\Delta$-point if and only if for every $\varepsilon>0$ and every slice $S$ with $m_{xy}\in S$ there exist $u, v\in M$ with $u\neq v$ such that $m_{uv}\in S$ and $d(u,v)<\varepsilon$ (see {\cite[Theorem~4.7]{JRZ}}). That result was later generalized for convex combinations of molecules (see {\cite[Theorem~4.4]{Vee}}). With the help of Proposition \ref{general} this result can be generalized to all unit sphere elements in Lipschitz-free spaces over proper metric spaces. First let us note that by Proposition \ref{general} and Remark \ref{remark_conv} we get the following.

\begin{prop}\label{lip_free_prop}
Let $\mu \in S_{\mathcal{F}(M)}$ be a $\Delta$-point. Then for every $\varepsilon>0$ and every slice $S$ of $B_{\mathcal{F}(M)}$ with $\mu\in S$ there exists a sequence $(m_{u_iv_i})$ in $S$ such that 
\[\|m_{u_iv_i}-m_{u_jv_j}\|\ge 2-\varepsilon\]
for all $i,j\in \N$ with $i\neq j$.
\end{prop}

\begin{prop}\label{delta_proper}
Assume that $M$ is proper and let $\mu \in S_{\mathcal{F}(M)}$. Then $\mu$ is a $\Delta$-point if and only if for every $\varepsilon>0$ and every slice $S$ of $B_{\mathcal{F}(M)}$ with $\mu\in S$ there exist $u, v\in M$ with $u\neq v$ such that $m_{uv}\in S$ and $d(u,v)<\varepsilon$.
\end{prop}

\begin{proof}
We only need to prove the "only if" part, the "if" part is a direct consequence of {\cite[Theorem~2.6]{JRZ}}. Assume that $\mu$ is a $\Delta$-point. Fix $\varepsilon>0$ and a slice $S(f,\alpha)$ of $B_{\mathcal{F}(M)}$ with $\mu\in S(f,\alpha)$. Choose $\gamma>0$ such that $\mu\in S(f,\alpha-\gamma)$, and let $\nu\in\conv\mathcal{M}(M)\cap S_{\mathcal{F}(M)}$ be such that $\|\mu-\nu\|<\gamma$. Let $n\in \mathbb{N}$, $\lambda_1,\ldots,\lambda_n>0$ with $\sum^{n}_{i=1}\lambda_i=1$, and $m_{x_1 y_1},\ldots,m_{x_n y_n}\in S_{\mathcal{F}(M)}$ be such that  $\nu=\sum^{n}_{i=1}\lambda_i m_{x_iy_i}$. Choose $g\in S_{\Lip_0(M)}$ such that $g(\nu)=1$ and define function $h\colon M\rightarrow \mathbb{R}$ by
\[h(p)=\max\Big\{\min_{i\in\{1,\ldots,n\}}g(y_i),\max_{i\in\{1,\ldots,n\}}\big(g(x_i)-d(x_i,p)\big)\Big\}+a\]
where $a\in \mathbb{R}$ is such that $h(0)=0$. From {\cite[Proposition~1.32]{Weaver}} we get $\|h\|\le1$. For every $i\in \{1,\ldots,n\}$ we have
\[h(x_i)\ge g(x_i)+a\]
and for some $j\in \{1,\ldots,n\}$ we have
\begin{align*}
    h(y_i)&= \max\Big\{\min_{k\in\{1,\ldots,n\}}g(y_k),g(x_j)-d(x_j,y_i)\Big\}+a\\
    &\le \max\big\{g(y_i),g(y_i)\big\}+a\\
    &=g(y_i)+a.
\end{align*}
Therefore $h(\nu)=1$ and 
\[f(\mu)+h(\mu)>1-\alpha+\gamma +g(\nu)-\gamma=2-\alpha.\]
By Proposition \ref{lip_free_prop} there exists a sequence
\[(m_{u_iv_i})\subseteq S\Big(\frac{f+h}{\|f+h\|},1-\frac{2-\alpha}{\|f+h\|}\Big)\]
such that $\|m_{u_iv_i}-m_{u_jv_j}\|\ge 1$ for all $i,j\in \N$ with $i\neq j$. Then $(m_{u_iv_i})\subseteq S(f,\alpha)$ and $(m_{u_iv_i})\subseteq S(h,\alpha)$. Note that
\[\min_{i\in\{1,\ldots,n\}}g(y_i)+a\le h(p)\le \max_{i\in\{1,\ldots,n\}}g(y_i)+a,\]
which means that the set $\big\{u,v\in M\colon m_{uv}\in S(h,\alpha)\big\}$ is bounded. Therefore there exist convergent subsequences of $(u_i)$ and $(v_i)$, which must converge to the same element, since the sequence $(m_{u_iv_i})$ does not have any convergent subsequences. This means that there exist $i\in\N$ such that $d(u_i,v_i)<\varepsilon$.
\end{proof}

Similarly to Proposition \ref{delta_proper} we can prove the following result.

\begin{prop}\label{delta_3}
Assume that $M$ is such complete metric space that for every $\varepsilon>0$ and every bounded subset $M_0\subseteq M$ the set
\[\big\{m_{uv}\in\dent(B_{\mathcal{F}(M)})\colon u,v\in M_0, d(u,v)>\varepsilon\big\}\]
is finite. Let $\mu\in S_{\mathcal{F}(M)}$. Then $\mu$ is a $\Delta$-point if and only if for every $\varepsilon>0$ and every slice $S$ of $B_{\mathcal{F}(M)}$ with $\mu\in S$ there exist $u, v\in M$ with $u\neq v$ such that $m_{uv}\in S$ and $d(u,v)<\varepsilon$.
\end{prop}

\begin{proof}
We only need to prove the "only if" part, the "if" part is a direct consequence of {\cite[Theorem~2.6]{JRZ}}. Assume that $\mu$ is a $\Delta$-point. Fix $\varepsilon>0$ and a slice $S(f,\alpha)$ of $B_{\mathcal{F}(M)}$ with $\mu\in S(f,\alpha)$. If there exist a subslice $S(g,\beta)\subseteq S(f,\alpha)$ such that $g$ is local, then there exist $m_{uv}\in S(g,\beta)\subseteq S(f,\alpha)$ such that $d(u,v)<\varepsilon$.

Now assume that for every subslice $S(g,\beta)\subseteq S(f,\alpha)$ the functional $g$ is not local.  Then analogously to the proof of Proposition \ref{delta_proper}, we can find bounded sequences $(u_i)$ and $(v_i)$ such that  $(m_{u_iv_i})\subseteq S(f,\alpha)$ and $\|m_{u_iv_i}-m_{u_jv_j}\|\ge 1$ for all $i,j\in \N$ with $i\neq j$. Furthermore, by Proposition \ref{local_dent} every subslice of $S(f,\alpha)$ contains a denting point of $B_{\mathcal{F}(M)}$ and by {\cite[Corollary~3.44]{Weaver}} that denting point is a molecule. Therefore by Remark \ref{remark_conv} we may assume  $(m_{u_iv_i})\subseteq  \dent (B_{\mathcal{F}(M)})$. This means that there exist $i\in\N$ such that $d(u_i,v_i)<\varepsilon$, since the set \[\big\{m_{uv}\in\dent(B_{\mathcal{F}(M)})\colon u,v\in \{u_j,v_j\colon j\in \N\}, d(u,v)\ge\varepsilon\big\}\]
is finite.
\end{proof}

Note that Proposition \ref{delta_3} also applies to metric spaces that are not proper, for example, to the metric space from  {\cite[Example~3.1]{Vee}}.

\section{Equivalence of Daugavet- and Delta-points in spaces of Lipschitz functions over proper metric spaces}
It has been shown that any local function $f\in S_{\Lip_0(M)}$ is a Daugavet-point (see {\cite[Theorem~3.6]{JRZ}}, {\cite[Theorem~1.4]{HOP}}). First we show that in spaces of Lipschitz functions over compact metric spaces only local functions are Daugavet-points. By Proposition \ref{general_weak} and Remark \ref{remark_conv} we get the following result.

\begin{prop}\label{prop_delta}
Let $f\in S_{\Lip_0(M)}$ be a $w^*$-$\Delta$-point. Then for every $\varepsilon>0$ there exists a sequence $(m_{u_iv_i})$ in $S(f,\varepsilon)$ such that 
\[\|m_{u_iv_i}-m_{u_jv_j}\|\ge 2-\varepsilon\]
for all $i,j\in \N$ with $i\neq j$.
\end{prop}

\begin{prop}\label{compact_delta}
Assume that $M$ is compact and let $f\in S_{\Lip_0(M)}$ be a $w^*$-$\Delta$-point. Then $f$ is local.
\end{prop}

\begin{proof}
Fix $\varepsilon>0$. We may assume $\varepsilon<2$. By Proposition \ref{prop_delta} there exists a sequence $(m_{u_iv_i})\subseteq S(f,\varepsilon)$  such that 
\[\|m_{u_iv_i}-m_{u_jv_j}\|\ge 2-\varepsilon\]
for all $i,j\in \N$ with $i\neq j$. Since $M$ is compact, by moving to a subsequence we may assume the sequences $(u_i)$ and $(v_i)$ are convergent. Let $u,v\in M$ be such that $u_i\rightarrow u$  and $v_i\rightarrow v$. If $u\neq v$, then $m_{u_iv_i}\rightarrow m_{uv}$, which is clearly a contradiction, since the sequence $(m_{u_iv_i})$ does not converge. Therefore $u=v$ and there exist $i\in \N$ such that $d(u_i,v_i)<\varepsilon$. By assumption we also have $f(m_{u_iv_i})>1-\varepsilon$ and therefore $f$ is local.
\end{proof}
Now we see that in the case of compact metric spaces, Proposition \ref{prop_delta} offers us a characterization. Our next step is to generalize that to proper metric spaces. To do so, we shall first present a more general version of {\cite[Proposition~4.2]{HOP}}. The proof of this result is somewhat similar to the proof of {\cite[Theorem~3.3]{HOP}}.

\begin{prop}\label{something}
Let $\varepsilon>0$ and let $(m_{u_iv_i})$ be a sequence in $\mathcal{F}(M)$ such that there exists a sequence $(A_i)$ of pairwise disjoint subsets of $M$ such that $u_i\in A_i$ and
\[d(u_i,x)+d(v_i,y)\ge (1-\varepsilon)\big(d(u_i,v_i)+d(x,y)\big)\]
for all $i\in \N$ and $x,y\in M\setminus A_i$. Then $\limsup \|F+ m_{u_iv_i}\|\ge 2-2\varepsilon$ for every $F\in S_{\Lip_0(M)^*}$.
\end{prop}

\begin{proof}
Fix  $F\in S_{\Lip_0(M)^*}$. 
By de Leeuw’s transform, there exists $\mu\in ba(\widetilde{M})$ with $|\mu|(\widetilde{M})=1$ where
\[\widetilde{M}=(M\times M)\setminus \{(x,x)\colon x\in M\}\]
such that
\[F(f)=\int_{\widetilde{M}} \tilde{f}d\mu\quad \text{where }\quad\tilde{f}(x,y)=\frac{f(x)-f(y)}{d(x,y)}\]
for every $f\in\Lip_0(M)$ (see, e.g., \cite{Weaver}). Following the notation used in \cite{HOP}, let
\[\Gamma_{1,A}=\{(x,y)\in \widetilde{M}\colon x\in A\}\quad\text{ and }\quad\Gamma_{2,A}=\{(x,y)\in \widetilde{M}\colon y\in A\}.\] Choose $\gamma>0$ arbitrarily small and $f\in S(F,\gamma)$. There exists $n\in \N$ such that $|\mu|(\Gamma_{1,A_i})<\gamma$ and $|\mu|(\Gamma_{2,A_i})<\gamma$ for every $i\ge n$, since the sets $A_1,A_2,\ldots$ are pairwise disjoint. 

Fix $i\ge n$. We will show that $\|F+ m_{u_iv_i}\|\ge 2-2\varepsilon-5\gamma$ and since $\gamma$ is arbitrarily small, this gives us $\limsup \|F+ m_{u_jv_j}\|\ge 2-2\varepsilon.$
Consider two cases.

\textbf{Case 1.} Assume $v_i\notin A_i$. Let us
define a function $g$ on $(M\setminus A_i)\cup \{u_i\}$ in such way that $g|_{M\setminus A_i}=(1-\varepsilon)f|_{M\setminus A_i}$ and 
\[g(u_i)=g(v_i)+(1-\varepsilon)d(u_i,v_i).\]
Note that for any $x\in M\setminus A_i$ we have
\begin{align*}
    \big|g(u_i)-g(x)\big|&= \big|(1-\varepsilon)f(v_i)+(1-\varepsilon)d(u_i,v_i)-(1-\varepsilon)f(x)\big|\\
    &\le (1-\varepsilon)\big(d(u_i,v_i)+|f(x)-f(v_i)|\big)\\
    &\le (1-\varepsilon)\big(d(u_i,v_i)+d(x,v_i)\big)\\
    &\le d(x,u_i).
\end{align*}
Therefore the Lipschitz constant of $g$ is at most 1. Let us extend $g$ by McShane--Whitney Theorem to $M$. Then
\begin{equation}\label{4_3_eq}
    F(g)=(1-\varepsilon)\int_{\widetilde{M}\setminus (\Gamma_{1,A_i}\cup \Gamma_{2,A_i})}\tilde{f}d\mu+\int_{ \Gamma_{1,A_i}\cup \Gamma_{2,A_i}}\tilde{g}d\mu>(1-\varepsilon)(1-3\gamma)-2\gamma>1-\varepsilon-5\gamma
\end{equation}
and 
\[g(m_{u_iv_i})=1-\varepsilon,\]
meaning $\|F+ m_{u_iv_i}\|\ge 2-2\varepsilon-5\gamma$.

\textbf{Case 2.} Assume $v_i\in A_i$. This case is very similar to the proof of {\cite[Theorem~3.3]{HOP}}. 
Let us
define a function $g\colon M\rightarrow\mathbb{R}$ in such way that $g|_{M\setminus A_i}=(1-\varepsilon)f|_{M\setminus A_i}$, 
\[g(u_i)=\inf_{x\in M\setminus A_i}\big(g(x)+d(x,u_i)\big),\] 
and 
\[g(y)=\sup_{x\in (M\setminus A_i)\cup\{u_i\}}\big(g(x)-d(x,y)\big)\]
for every $y\in A_i\setminus\{u_i\}$. Then $\|g\|\le 1$ and either $g(v_i)=g(u_i)-d(u_i,v_i)$ or
\begin{align*}
    g(u_i)-g(v_i)&=\inf_{x,y\in M\setminus A_i}\big(g(x)+d(x,u_i)-g(y)+d(y,v_i)\big)\\
    &\ge\inf_{x,y\in M\setminus A_i}\big(d(x,u_i)+d(y,v_i)-(1-\varepsilon)d(x,y)\big)\\
    &\ge (1-\varepsilon)d(u_i,v_i),
\end{align*}
giving us $g(m_{u_iv_i})\ge1-\varepsilon$. Furthermore, for $g$ the inequalities \eqref{4_3_eq} hold, and we get $F(g)>1-\varepsilon-5\gamma$. Hence $\|F+ m_{u_iv_i}\|\ge 2-2\varepsilon-5\gamma$. 
\end{proof}

Our main interest in Proposition \ref{something} was to show the following.
\begin{prop}\label{unbounded_prop}
Let a sequence $(m_{u_iv_i})$ in $\mathcal{F}(M)$ be such that at least one of the sequences $(u_i)$ and  $(v_i)$ is unbounded. Then $\limsup \|F+m_{u_iv_i}\|=2$ for every $F\in S_{\Lip_0(M)^*}$.
\end{prop}

In order to prove this, we shall first prove the following lemma.
\begin{lem}\label{lemma_unbounded}
Let $\varepsilon>0$. Then
\[d(u,x)+d(v,y)\ge (1-\varepsilon)\big(d(u,v)+d(x,y)\big)\]
for all $a>0$, $u\in B(0,8a)$, $v\in B(0,8a)\setminus B(0,4a)$, and $x,y\in \big(M\setminus B(0,32a/\varepsilon)\big)\cup B(0,a\varepsilon)$.
\end{lem}

\begin{proof}
Fix $a>0$, $u\in B(0,8a)$, and $v\in B(0,8a)\setminus B(0,4a)$.
If  $x,y\in B\big(0,a\varepsilon\big)$, then 
\[d(v,y)\ge d(v,0)-d(0,y)>4a-a\varepsilon\ge 4a(1-\varepsilon)\frac{d(x,y)}{2a\varepsilon}=2\frac{1-\varepsilon}{\varepsilon}d(x,y)\]
and we get
\begin{align*}
    d(u,x)+d(v,y)&\ge (1-\varepsilon)\big(d(u,v)-d(x,y)\big)+\varepsilon \big(d(u,x)+d(v,y)\big)\\
    &\ge(1-\varepsilon)\big(d(u,v)-d(x,y)\big)+2(1-\varepsilon)d(x,y)\\
    &= (1-\varepsilon)\big(d(u,v)+d(x,y)\big).
\end{align*}
Furthermore,
\begin{align*}
    d(u,x)+d(v,y)&\ge d(x,0)-d(0,u)+d(y,0)-d(0,v)\\
    &\ge(1-\varepsilon)d(x,y)+\varepsilon\big(d(x,0)+d(y,0)\big)-d(0,u)-d(0,v)\\
    &>(1-\varepsilon)d(x,y)+ 32a-d(0,u)-d(0,v) \\
    &\ge (1-\varepsilon)\big(d(u,v)+d(x,y)\big)
\end{align*}
when $x\in M\setminus B(0,32a/\varepsilon)$ or $y\in M\setminus B(0,32a/\varepsilon)$.
\end{proof}

\begin{proof}[Proof of Proposition \ref{unbounded_prop}]
We will find a sequence $(k_i)\subseteq\N$ such that there exists a sequence $(A_i)$ of pairwise disjoint subsets of $M$ such that $u_{k_i}\in A_i$ or $v_{k_i}\in A_i$ and
\[d(u_{k_i},x)+d(v_{k_i},y)\ge (1-1/i)\big(d(u_{k_i},v_{k_i})+d(x,y)\big)\]
for all $i\in \N$ and $x,y\in M\setminus A_i$. Then by Proposition \ref{something} we get $\limsup \|F+m_{u_{k_i}v_{k_i}}\|=2$ and therefore $\limsup \|F+m_{u_iv_i}\|=2$ for every $F\in S_{\Lip_0(M)^*}$.

Let $k_1=1$ and choose $a_1>0$ such that $u_{k_1},v_{k_1}\in B(0,8a_1)$ and either $u_{k_1}\notin B(0,4a_1)$ or $v_{k_1}\notin B(0,4a_1)$. Let $A_1=B(0,32a_1)\setminus B(0,a_1)$. By Lemma \ref{lemma_unbounded} we have
\[d(u_{k_1},x)+d(v_{k_1},y)\ge (1-1/i)\big(d(u_{k_1},v_{k_1})+d(x,y)\big)\]
for all $x,y\in M\setminus A_1$.

Assume that we have found $k_1,\ldots,k_{n-1}\in\N$, and  pairwise disjoint bounded sets $A_1,\ldots,A_{n-1}$ such that $u_{k_i}\in A_i$ or $v_{k_i}\in A_i$ and
\[d(u_{k_i},x)+d(v_{k_i},y)\ge (1-1/i)\big(d(u_{k_i},v_{k_i})+d(x,y)\big)\]
for all $i\in \{1,\ldots,n-1\}$ and $x,y\in M\setminus A_i$. Choose $a_n>0$ such that there exists $i\in\N$ such that $u_i,v_i\in B(0,8a_n)$ with either $u_i\notin B(0,4a_n)$ or $v_i\notin B(0,4a_n)$, and $\cup_{j=1}^{n-1}A_j\subseteq B(0,a_n/n)$. Let $k_n=i$ and $A_n=B(0,32a_nn)\setminus B(0,a_n/n)$. Clearly, $A_n\cap A_j=\emptyset$ for every $j\in\{1,\ldots,n-1\}$ and by Lemma \ref{lemma_unbounded} we have
\[d(u_{k_n},x)+d(v_{k_n},y)\ge (1-1/n)\big(d(u_{k_n},v_{k_n})+d(x,y)\big)\]
for every $x,y\in M\setminus A_n$.

\end{proof}

By combining these results we get the following sufficient condition for Daugavet-points.
\begin{prop}\label{proper_delta}
Let $f\in S_{Lip_0(M)}$ be such that for every $\varepsilon>0$ there exists a sequence $(m_{u_iv_i})$ in $S(f,\varepsilon)$ and a sequence $(A_i)$ of pairwise disjoint subsets of $M$ such that $u_i\in A_i$ and
\[d(u_i,x)+d(v_i,y)\ge (1-\varepsilon)\big(d(u_i,v_i)+d(x,y)\big)\]
for all $i\in \N$ and $x,y\in M\setminus A_i$. Then $f$ is a Daugavet-point.

In particular, if $f\in S_{Lip_0(M)}$ is such that for every $\varepsilon>0$ there exists a sequence $(m_{u_iv_i})$ in $S(f,\varepsilon)$ such that at least one of the sequences $(u_i)$ and  $(v_i)$ is unbounded, then $f$ is a Daugavet-point.
\end{prop}

Now let us present the main result of this section.
\begin{thm}\label{main_thm4}
Assume that $M$ is proper and let $f\in S_{\Lip_0(M)}$. The following are equivalent:
\begin{enumerate}
    \item $f$ is a Daugavet-point;
    \item $f$ is a $\Delta$-point;
    \item $f$ is a $w^*$-Daugavet-point;
    \item $f$ is a $w^*$-$\Delta$-point;
    \item for every $\varepsilon>0$ there exists a sequence $(m_{u_iv_i})\subseteq S(f,\varepsilon)$ such that 
    \[\|m_{u_iv_i}-m_{u_jv_j}\|\ge 2-\varepsilon\]
    for every $i,j\in\N$ with $i\neq j$;
    \item for every $\varepsilon>0$ there exists a sequence $(m_{u_iv_i})\subseteq S(f,\varepsilon)$ such that either $d(u_i,v_i)\rightarrow0$ or at least one of the sequences $(u_i)$ and  $(v_i)$ is unbounded.
\end{enumerate}
\end{thm}

\begin{proof}
We see that $(1)\Rightarrow(2)$, $(2)\Rightarrow(4)$, $(1)\Rightarrow(3)$ and $(3)\Rightarrow(4)$ are clear from definitions, $(4)\Rightarrow(5)$ is Proposition \ref{prop_delta} and $(6)\Rightarrow(1)$ is {\cite[Theorem~1.4]{HOP}} and Proposition \ref{proper_delta}. 

The only thing left to prove is $(5)\Rightarrow(6)$. Assume that for $\varepsilon>0$ there exists a sequence $(m_{u_iv_i})\subseteq S(f,\varepsilon)$ such that 
\[\|m_{u_iv_i}-m_{u_jv_j}\|\ge 2-\varepsilon.\]
If both sequences $(u_n)$ and $(v_n)$ are bounded, then analogously to the proof of Proposition \ref{compact_delta} by compactness by moving to subsequences we get  $d(u_n,v_n)\rightarrow0$.
\end{proof}

As a last part of this section we prove the equivalence of Daugavet- and $\Delta$-points for a different class of spaces of Lipschitz functions. By definition, if $f\in S_{\Lip_0(M)}$ is not local, then there exists $\varepsilon>0$ such that $m_{uv}\in S(f,\varepsilon)$ only if $d(u,v)>0$, and therefore from Proposition \ref{local_dent} we get that every subslice $S(g,\alpha)\subseteq S(f,\varepsilon)$ contains a denting point, since its defining functional $g$ is also not local.
Then by Remark \ref{remark_conv} and Proposition \ref{general_weak} we get the following.

\begin{prop}\label{prop_dent_delta}
Let $f\in S_{\Lip_0(M)}$ be a $w^*$-$\Delta$-point that is not local. Then for every $\varepsilon>0$ there exists a sequence $(\mu_i)$ in  $S(f,\varepsilon)\cap \dent(B_{\mathcal{F}(M)})$ such that 
\[\|\mu_i-\mu_j\|\ge 2-\varepsilon\]
for all $i,j\in \N$ with $i\neq j$.
\end{prop}

Recall that if $M$ is complete, then all denting points are molecules (see {\cite[Corollary~3.44]{Weaver}}). The proof of the next theorem is analogous to the proof of Theorem \ref{main_thm4} with the exception that we use Proposition  \ref{prop_dent_delta} instead of Proposition \ref{prop_delta}.
\begin{thm}
Assume that $M$ is such complete metric space that for every $\varepsilon>0$ and every bounded subset $M_0\subseteq M$, the set
\[\big\{m_{uv}\in\dent(B_{\mathcal{F}(M)})\colon u,v\in M_0, d(u,v)>\varepsilon\big\}\]
is finite. Let $f\in S_{\Lip_0(M)}$. The following are equivalent:
\begin{enumerate}\label{main_thm4_extra}
    \item $f$ is a Daugavet-point;
    \item $f$ is a $\Delta$-point;
    \item $f$ is a $w^*$-Daugavet-point;
    \item $f$ is a $w^*$-$\Delta$-point;
    \item for every $\varepsilon>0$ there exists a sequence $(m_{u_iv_i})\subseteq S(f,\varepsilon)$ such that either $d(u_i,v_i)\rightarrow0$ or at least one of the sequences $(u_i)$ and  $(v_i)$ is unbounded.
\end{enumerate}
\end{thm}

\begin{proof}
We only need to show $(4)\Rightarrow(5)$. Assume $f$ is a $w^*$-$\Delta$-point and fix $\varepsilon>0$. If $f$ is local, then for every $\varepsilon>0$ there exist $(m_{u_iv_i})\subseteq S(f,\varepsilon)$ such that $d(u_i,v_i)\rightarrow0$. Otherwise by Proposition \ref{prop_dent_delta} we have a sequence $(m_{u_iv_i})\subseteq S(f,\alpha)\cap \dent(B_{\mathcal{F}(M)})$ such that 
\[\|m_{u_iv_i}-m_{u_jv_j}\|\ge 2-\varepsilon\]
for every $i,j\in \N$ with $i\neq j$. Clearly, the molecules $m_{u_iv_i}$ are pairwise distinct if $\varepsilon<2$, which we may assume. If both $(u_n)$ and $(v_n)$ are bounded, then by assumption for every $\gamma>0$ there is a finite number of denting points $m_{uv}$ such that $u\in \{u_i\colon i\in\N\}$, $v\in \{v_i\colon i\in\N\}$ and $d(u,v)>\gamma$. Therefore
$d(u_i,v_i)\rightarrow0$.
\end{proof}

As was noted in the end of Section 3, we can also apply Theorem \ref{main_thm4_extra} to metric spaces that are not proper, for example, to the metric space from  {\cite[Example~3.1]{Vee}}.

\section{Existence of Daugavet- and Delta-points in spaces of Lipschitz functions}
In this section we prove that every space of Lipschitz functions over an infinite metric space contains a $\Delta$-point. Furthermore, we show that the same does not hold for Daugavet-points, thus showing that there exist $\Delta$-points in spaces of Lipschitz functions that are not Daugavet-points.

We start by showing that if $M$ satisfies certain conditions, for example, if $M$ is unbounded or not uniformly discrete, then there exists a Daugavet-point in $\Lip_0(M)$.

\begin{prop}\label{Daug_existance}
Let $(m_{u_iv_i})$ be a sequence in $\mathcal{F}(M)$ such that there exists a sequence $(A_i)$ of pairwise disjoint subsets of $M$ such that $u_i\in A_i$ and
\[d(u_i,x)+d(v_i,y)\ge (1-1/2^{i+1})\big(d(u_i,v_i)+d(x,y)\big)\]
for all $i\in \N$ and $x,y\in M\setminus A_i$. Then there exists a Daugavet-point $f\in S_{\Lip_0(M)}$.
\end{prop}

\begin{proof}
By moving to a subsequence we may assume that
\[A_i\cap \big\{u_j,v_j\colon j\in \{1,\ldots,i-1\}\big\}=\emptyset\]
for every $i\in\N$, since $(A_i)$ are pairwise disjoint.
For simplicity we will assume $u_1$ is the point $0$. 
Let us define $f$ on the set $\{u_i,v_i\colon i\in\N\}$ recursively. Let 
\[f(u_1)=f(v_1)=0.\]
Assume we have defined $f(u_1),\ldots,f(u_{n-1}),f(v_1),\ldots,f(v_{n-1})$ in such way that $f(m_{u_iv_i})\ge 1-1/2^{i-1}$ for every $i\in\{1,\ldots,n-1\}$ and the Lipschitz constant of $f$ is at most $1-1/2^{n-1}$. Let $M_{n-1}=\big\{u_i,v_i\colon i\in\{1,\ldots,n-1\}\big\}$ and let 
\[f(u_n)=\min_{x\in M_{n-1}}\big(f(x)+(1-1/2^n)d(x,u_n)\big)\]
and 
\[f(v_n)=\max_{x\in M_{n-1}\cup\{u_n\}}\big(f(x)-(1-1/2^n)d(x,v_n)\big).\]
Then the Lipschitz constant of $f$ is at most $1-1/2^n$ and either $f(v_n)=f(u_n)-(1-1/2^n)d(u_n,v_n)$ or
\begin{align*}
    f(u_n)-f(v_n)&=\min_{x,y\in M_{n-1}}\big(f(x)+(1-1/2^n)d(x,u_n)-f(y)+(1-1/2^n)d(y,v_n)\big)\\
    &\ge\min_{x,y\in M_{n-1}}\big((1-1/2^n)\big(d(x,u_n)+d(y,v_n)\big)-(1-1/2^{n-1})d(x,y)\big)\\
    &\ge(1-1/2^n)\min_{x,y\in M_{n-1}}\big(d(x,u_n)+d(y,v_n)-(1-1/2^{n+1})d(x,y)\big)\\
    &\ge (1-1/2^n)(1-1/2^{n+1})d(u_n,v_n)\\
    &\ge (1-1/2^{n-1})d(u_n,v_n),
\end{align*}
giving us $f(m_{u_nv_n})\ge1-1/2^{n-1}$. Extend $f$ by McShane--Whitney Theorem to $M$. Then by Proposition \ref{proper_delta} $f$ is a Daugavet-point.
\end{proof}

If $M$ is either unbounded or not uniformly discrete, then by following the idea of the proof of  {\cite[Proposition~2.4]{HOP}}, it is possible to construct a sequence $(m_{u_iv_i})$ in $\mathcal{F}(M)$ and a sequence $(A_i)$ of pairwise disjoint subsets of $M$ such that $u_i\in A_i$ and
\[d(u_i,x)+d(v_i,y)\ge (1-1/2^{i+1})\big(d(u_i,v_i)+d(x,y)\big)\]
for all $i\in \N$ and $x,y\in M\setminus A_i$.
Thus we get the following.

\begin{cor}\label{Daug_existance_cor}
Assume that $M$ is either unbounded or not uniformly discrete. Then $\Lip_0(M)$ contains a Daugavet-point.
\end{cor}

Next we will show the existence of Daugavet-points under different conditions than in Proposition \ref{Daug_existance}. First let us provide another sufficient condition for Daugavet-points.

\begin{prop}\label{sufficient_daug}
Let $f\in S_{\Lip_0(M)}$, and assume that there exists a sequence $(u_i)$ of pairwise distinct points in $M$ such that for all $i\in\N$, $\delta>0$, and $v\in M\setminus\{u_i\}$ there exists $p\in[u_i,v]_\delta\setminus\{u_i\}$ with $m_{u_ip}\in S(f,\delta)$. Then $f$ is a Daugavet-point.
\end{prop}

\begin{proof}
If $f$ is local, then it is a Daugavet-point, therefore we only need to consider the case when $f$ is not local. Fix $g\in S_{\Lip_0(M)}$ and $\varepsilon>0$. We will show that $g\in\clconv \Delta_\varepsilon(f)$.
For every $i\in\N$ let us
define a function $h_i$ on $M$ in such way that $h_i|_{M\setminus \{u_i\}}=g|_{M\setminus \{u_i\}}$ and 
\[h_i(u_i)=\sup_{v\in M\setminus \{u_i\}}\big(g(v)-d(v,u_i)\big).\]
Then $\|h_i\|\le 1$. Since $f$ is not local, there exists $\delta>0$ such that $d(u,v)>\delta$ for every $m_{uv}\in S(f,\delta)$. We may also assume that $2\delta<\varepsilon$ and $\delta<1$. There exists $v\in M\setminus \{u_i\}$ such that
\[h_i(u_i)<g(v)-d(v,u_i)+\delta^2/2.\] Therefore 
\[h_i(v)-h_i(u_i)>g(v)-\big(g(v)-d(v,u_i)+\delta^2/2\big)=d(v,u_i)-\delta^2/2.\]
There exists $p\in[u_i,v]_{\delta^2/2}$ such that $m_{u_ip}\in S(f,\delta^2/2)$. Since $\delta^2/2<\delta$, we get $d(u_i,p)>\delta$. Hence
\begin{align*}
    h_i(p)-h_i(u_i)&=h_i(p)-h_i(v)+h_i(v)-h_i(u_i)\\
    &>-d(p,v)+d(v,u_i)-\delta^2/2\\
    &>d(p,u_i)-\delta^2\\
    &>(1-\delta)d(p,u_i).
\end{align*}
Then for every $i\in \N$ we get
\[\|f-h_i\|\ge f(m_{u_ip})-h_i(m_{u_ip})>2-2\delta>2-\varepsilon\]
and since $h_j$ and $f$ can be different only in $u_j$, we also get
\[\Big\|f-\frac{1}{i}\sum_{j=1}^{i}h_{j} \Big\|\le \frac{4}{i}.\]
Therefore $g\in \clconv\Delta_\varepsilon(f)$, which means that $f$ is a Daugavet-point.
\end{proof}

\begin{prop}\label{Daug_existance_2}
Let $(u_i)$ be a sequence of pairwise distinct points in $M$ such that for all $i\in\N$, $\delta>0$, and $v\in M\setminus\{u_i\}$ there exists $p\in[u_i,v]_\delta\setminus\{u_i\}$ with $\inf_{j\in\N}d(u_j,p)>(1-\delta)d(u_i,p)$.
Then there exists a Daugavet-point $f\in S_{\Lip_0(M)}$.
\end{prop}

\begin{proof}
For simplicity assume that $u_1$ is the fixed point $0$. Define $f$ by
\[f(p)=\inf_{i\in\N}d(u_i,p).\]
Then $\|f\|\le1$ by {\cite[Proposition~1.32]{Weaver}}. Furthermore, for all $i\in\N$, $\delta>0$, and $v\in M\setminus\{u_i\}$ there exists $p\in[u_i,v]_\delta\setminus\{u_i\}$ with $\inf_{j\in\N}d(u_j,p)>(1-\delta)d(u_i,p)$. Hence
\[f(p)-f(u_i)=\inf_{j\in\N}d(u_j,p)>(1-\delta)d(u_i,p).\]
By Proposition \ref{sufficient_daug}, $-f$ is a Daugavet-point. Therefore $f$ is also a Daugavet-point.
\end{proof}

Proposition \ref{Daug_existance_2} can be applied for metric spaces where Proposition \ref{Daug_existance} can not be applied. For example, let $M$ be infinite with the metric 
\[d(p,q)=\begin{cases}
        1, & p\in\{x,y\},q\notin\{x,y\}\text{ or }p\notin\{x,y\},q\in\{x,y\},\\
        2, & \text{otherwise}
        \end{cases}\]
where $x,y\in M$ are two fixed points. Then there does not exist a sequence $(m_{u_iv_i})$  in $\mathcal{F}(M)$ satisfying the conditions of Proposition \ref{Daug_existance}. However, any sequence $(u_i)$ of pairwise distinct points in $M\setminus\{x,y\}$ satisfies the conditions of Proposition \ref{Daug_existance_2}, thus there exists a Daugavet-point in $\Lip_0(M)$.

Next we will show the existence of $\Delta$-points in spaces of Lipschitz functions over infinite metric spaces. First we will introduce two helpful lemmas.

\begin{lem}\label{lem_equal_existance}
Assume that $M$ is infinite, bounded, and uniformly discrete. There exist $a>0$ and a sequence $(u_i)\subseteq M$ such that
\[a\frac{i-1}{i}\le d(u_i,u_j)\le a\frac{i+1}{i}\]
for all $i,j\in\N$ with $i<j$.
\end{lem}

\begin{proof}
We will construct the sequence $(u_i)$ recursively.
Choose $u_1\in M$ and let
\[a_1=\sup \Big\{b\in\mathbb{R}\colon\big|\big\{p\in M\colon d(u_1,p)\le b\big\}\big|< \infty\Big\}.\]
Clearly, $0<a_1<\infty$, since $M$ is bounded and uniformly discrete.
Let
\[M_1=\big\{p\in M\colon a_1(2-1)/2<d(u_1,p)<a_1(2+1)/2\big\}.\]
By the definition of $a_1$, the set $M_1$ is infinite. 

Assume that we have found $u_1,\ldots,u_{n-1}\in M$, $a_1,\ldots,a_{n-1}>0$, and infinite sets $M_{n-1}\subseteq\ldots\subseteq M_1\subseteq M$ such that 
\begin{equation}\label{eq_6}
    a_i\frac{2i-1}{2i}\le d(u_i,p)\le a_i\frac{2i+1}{2i}
\end{equation}
for all $i\in\{1,\ldots,n-1\}$ and $p\in M_i$.
Now let $u_{n}\in M_{n-1}$, let
\[a_{n}=\sup \Big\{b\in\mathbb{R}\colon\big|\big\{p\in M_{n-1}\colon d(u_{n},p)\le b\big\}\big|< \infty\Big\},\]
and let
\[M_{n}=\big\{p\in M_{n-1}\colon a_{n}(2n-1)/(2n)<d(u_{n},p)<a_{n}(2n+1)/(2n)\big\}.\]
As before, $0<a_{n}<\infty$ and $M_{n}\subseteq M_{n-1}$ is infinite.

Clearly, the sequence $(a_i)$ is positive and bounded, since $M$ is bounded and uniformly discrete. Therefore, by moving to a subsequence, we may assume there exists $a>0$ such that 
\[\frac{2i-2}{2i-1}a<a_i<\frac{2i+2}{2i+1}a\]
for every $i\in\N$. Note that moving to a subsequence the condition \eqref{eq_6} remains true. Then for all $i,j\in\N$ with $i<j$ we have 
\[a\frac{i-1}{i}<a_i\frac{2i-1}{2i}\le d(u_i,u_j)\le a_i\frac{2i+1}{2i}< a\frac{i+1}{i},\]
since $u_j\in M_{i}$.
\end{proof}

\begin{lem}\label{lem_delta_existance}
Assume that $M$ is infinite, bounded, and uniformly discrete. There exist $a>0$ and two sequences $(u_i)$ and $(v_i)$ of pairwise distinct points in $M$ such that 
\[a\frac{i+1}{i}\ge d(u_i,v_i)\ge a\frac{i-1}{i}\]
and
\[\min\big\{d(u_i,p),d(v_i,p)\big\}\ge a\frac{i-1}{i}\]
and
\[\min\big\{d(u_i,q),d(v_i,q)\big\}\ge a\frac{i-1}{2i}\]
for all $i\in\N$, $p\in \{u_j,v_j\colon j>i\}$, and $q\in M\setminus \{u_i,v_i\}$.
\end{lem}

\begin{proof}
Set 
\[a_1=\sup \Big\{b\in\mathbb{R}\colon\Big|\bigcup_{p\in M} \big(B(p,b)\setminus\{p\}\big)\Big|< \infty\Big\}.\]
Clearly, $0<a_1<\infty$, since $M$ is bounded and uniformly discrete. We will construct recursively two sequences $(x_i)$ and $(y_i)$ in $M$. Choose 
\[x_1\in \bigcup_{p\in M} \big(B(p,a_1+a_1)\setminus\{p\}\big)\]
and $y_1\in M$
such that $0<d(x_1,y_1)\le a_1+a_1$. Assume that we have found $x_1,\ldots,x_{n-1}\in M$ and $y_1,\ldots,y_{n-1}\in M$. Then choose 
\[x_n\in \bigcup_{p\in M} \big(B(p,a_1+a_1/n)\setminus\{p\}\big)\setminus \Big(\bigcup_{p\in M} \big(B(p,a_1-a_1/n)\setminus\{p\}\big)\cup\bigcup_{i=1}^{n-1}\{x_i,y_i\}\Big)\]
and $y_n\in M$
such that $0<d(x_n,y_n)\le a_1+a_1/n$. Clearly, $(x_i)$ is pairwise distinct. By moving to a subsequence we either have $(y_i)$ pairwise distinct or constant. Let us consider these cases separately.
If $(y_i)$ is pairwise distinct, then by moving to a subsequence we may also assume that
\[y_i\notin \bigcup_{p\in M} \big(B(p,a_1-a_1/i)\setminus\{p\}\big)\cup\bigcup_{j=1}^{i-1}\{x_j,y_j\}\]
for every $i\in\N$. Then
we may choose $a$ as $a_1$, $(u_i)$ as $(x_i)$, and $(v_i)$ as $(y_i)$. 

Now assume that $(y_i)$ is constant, i.e., there exists $z\in M$ such that $z=y_i$ for every $i\in \N$. Let
\[M_0=\{x_i\colon i\in\N\}.\]
By Lemma \ref{lem_equal_existance} there exist $a>0$ and a sequence $(y_i)\subseteq M_0$ such that
\[a\frac{i-1}{i}\le d(y_i,y_j)\le a\frac{i+1}{i}\]
for every $i,j\in\N$ with $i<j$.
Note that for any $i,j\in\N$ we have 
\[d(x_i,x_j)\le d(x_i,z)+d(x_j,z)\le a_1\frac{i+1}{i}+a_1\frac{j+1}{j}\]
and therefore $a\le2a_1$. We may choose $(u_i)$ as  $(y_{2i-1})$ and $(v_{i})$ as $(y_{2i})$. 
\end{proof}

\begin{thm}\label{Delta_existance}
If $M$ is an infinite metric space, then $\Lip_0(M)$ contains a $\Delta$-point.
\end{thm}

\begin{proof}
If $M$ is unbounded or not uniformly discrete, then by Corollary \ref{Daug_existance_cor} there exists a Daugavet-point in $\Lip_0(M)$.

Assume that $M$ is bounded and uniformly discrete. By Lemma \ref{lem_delta_existance} there exist $a>0$ and two sequences $(u_i)$ and $(v_i)$ of pairwise distinct points in $M$ such that 
\[a\frac{i+1}{i}\ge d(u_i,v_i)\ge a\frac{i-1}{i}\]
and
\[\min\big\{d(u_i,p),d(v_i,p)\big\}\ge a\frac{i-1}{i}\]
and
\[\min\big\{d(u_i,q),d(v_i,q)\big\}\ge a\frac{i-1}{2i}\]
for all $i\in\N$, $p\in \{u_j,v_j\colon j>i\}$ and $q\in M\setminus \{u_i,v_i\}$. For simplicity assume that $u_1$ is the fixed point $0$.
Let
\[f(p)=\begin{cases}
        a(i-2)/(2i), & p=u_i\text{ for }i\in\N\setminus\{1\},\\
        -a(i-2)/(2i), & p=v_i\text{ for }i\in\N\setminus\{1\},\\
        0, & \text{otherwise}.
        \end{cases}\]
Then $\|f\|\le 1$ and
\[f(m_{u_iv_i})=\frac{a(i-2)}{id(u_i,v_i)}\ge \frac{i-2}{i+1},\]
which gives us $f\in S_{\Lip_0(M)}$. Now for every $i\in\N\setminus\{1\}$, let 
\[g_i(p)=\begin{cases}
        a(j-2)/(2j), & p=u_j\text{ for }j\in\N\setminus\{1,i\},\\
        -a(j-2)/(2j), & p=v_j\text{ for }j\in\N\setminus\{1,i\},\\
        -a(i-2)/(2i), & p=u_i,\\
        a(i-2)/(2i), & p=v_i,\\
        0, & \text{otherwise}.
        \end{cases}\]
Analogously to previous argumentation $g_i\in S_{\Lip_0(M)}$. Furthermore, for every $i,j\in\N\setminus\{1\}$ we have
\[\|f-g_i\|\ge f(m_{u_iv_i})-g_i(m_{u_iv_i})\ge2\frac{i-2}{i+1}\]
and since $h_k$ and $f$ can be different only in $u_k$ and $v_k$, we also get
\[\Big\|f-\frac{1}{i}\sum_{k=1}^{i}g_{k+j} \Big\|\le \frac{4}{i}.\]
Therefore $f\in \clconv\Delta_\varepsilon(f)$  for every $\varepsilon>0$, i.e., $f$ is a $\Delta$-point.
\end{proof}

\begin{rem}
The existence of $\Delta$-points in spaces of Lipschitz functions over infinite metric spaces can also be derived from the fact that such spaces always contain an isometric copy of $\ell_\infty$ (see {\cite[Theorem~5]{CJ}}). Indeed, $\ell_\infty$ contains many $\Delta$-points and if some subspace $Y$ of a Banach space $X$ contains a $\Delta$-point $y$, then $y$ is also a $\Delta$-point in $X$, since the set of points in $Y$ that are almost at distance $2$ from $y$ is contained in the set of points in $X$ that are almost at distance $2$ from $y$.
\end{rem}

It is natural to wonder if every space of Lipschitz functions over an infinite metric space $M$ also contains a Daugavet-point or a $w^*$-Daugavet-point. We know that it is the case if $M$ is unbounded or not uniformly discrete (see Corollary \ref{Daug_existance_cor}). However, the following example shows that there does exist an infinite-dimensional space of Lipschitz functions that does not contain any $w^*$-Daugavet-points.

\begin{eks}
Consider $M=\N$ with the metric
\[d(n,k)=3-\Big|\frac{1}{n}-\frac{1}{k}\Big|.\]
Fix $f\in S_{\Lip_0(M)}$. We will show that $f$ is not a $w^*$-Daugavet-point.
There does not exist $n,k\in M$ such that $f(m_{1n})>3/4$ and $f(m_{k1})>3/4$, because if such elements would exist, we would get
\[f(k)-f(n)=f(k)-f(1)+f(1)-f(n)>\frac{3}{4}\big(d(1,k)+d(1,n)\big)>3>d(k,n).\]

As $f$ and $-f$ are $w^*$-Daugavet-points at the same time, we may assume that there exists $n\in M\setminus\{1\}$ such that $f(m_{1n})\ge0$ and also $f(m_{k1})\le3/4$ for every $k\in M\setminus\{1\}$. Assume that $n\in M$ is the smallest element in $M\setminus\{1\}$ such that $f(m_{1n})\ge0$.

Consider $\mu=\sum_{i=1}^{n-1}m_{in}/(n-1)$. Clearly, $\mu\in S_{\mathcal{F}(M)}$.
Let
\[\alpha=\frac{1}{3n}-\frac{1}{3(n+1)}\]
and let $g\in S_{\mathcal{F}(M)}$ be such that $\|f-g\|>2-\alpha$. We will show that $g\notin S(\mu,\alpha/n)$. There exist $k,l\in M$ with $k\neq l$ such that $f(m_{kl})>1-\alpha$ and $g(m_{lk})>1-\alpha$. Then $l>1$, since $f(m_{kl})>3/4$. Assume $l< n$. Then $k\neq 1$. If $k=2$, then
\[f(1)-f(l)>f(1)-f(k)+(1-\alpha)d(k,l)\ge-2-\frac{1}{2}+(1-\alpha)\Big(3-\frac{1}{2}+\frac{1}{l}\Big)>0,\]
otherwise
\[f(1)-f(l)>f(1)-f(k)+(1-\alpha)d(k,l)>-2-\frac{1}{3}+(1-\alpha)\Big(3-\frac{1}{2}\Big)>0.\]
This is a contradiction and therefore $l\ge n$. Also $k\neq n$, because if $l\neq n$, then
\[f(n)-f(l)\le f(1)-f(l)\le 2+\frac{1}{l}\le3-\frac{1}{n}+\frac{1}{l}-\frac{1}{2}<(1-1/6)d(n,l).\]
Consider three cases.

\textbf{Case 1.}  Assume $1\le k< n$. Then 
\begin{align*}
    g(k)-g(n)&<g(l)-g(n)-(1-\alpha)d(l,k)\\
    &\le d(l,n)-(1-\alpha)d(l,k)\\
    &< 3-2(1-\alpha)\\
    &<2(1-\alpha)\\
    &< (1-\alpha)d(k,n).
\end{align*}

\textbf{Case 2.}  Assume $l=n$. Then analogously to Case 1, we get
\begin{align*}
    g(1)-g(n)&<g(1)-g(k)-(1-\alpha)d(n,k)< (1-\alpha)d(1,n).
\end{align*}

\textbf{Case 3.} Assume $k,l>n$. Then 
\begin{align*}
    g(1)-g(n)&<g(1)-g(k)+g(l)-g(n)-(1-\alpha)\Big(3-\Big|\frac{1}{k}-\frac{1}{l}\Big|\Big)\\
    &< 2+\frac{1}{k}+3-\frac{1}{n}+\frac{1}{l}-3+\Big|\frac{1}{k}-\frac{1}{l}\Big| +3\alpha\\
    &= 2-\frac{1}{n}+\frac{1}{k}+\frac{1}{l}+\Big|\frac{1}{k}-\frac{1}{l}\Big|+\frac{2}{n}-\frac{2}{n+1} -3\alpha \\
    &\le 2+\frac{1}{n}-3\alpha\\
    &\le (1-\alpha)d(1,n).
\end{align*}
Therefore $g\notin S(m_{in},\alpha)$ for some $i\in\{1,\ldots,n-1\}$.
This means that $g\notin S(\mu,\alpha/n)$, since $S(\mu,\alpha/n)\subseteq S(m_{in},\alpha)$. Therefore $f$ is not a $w^*$-Daugavet-point.
\end{eks}

By Theorem \ref{Delta_existance} there exists a $\Delta$-point in the space of Lipschitz functions introduced in the previous example; therefore we now know that Daugavet- and $\Delta$-points are not equivalent in spaces of Lipschitz functions. We will finish this paper by showing that Daugavet- and $\Delta$-points also differ from their $w^*$-versions in spaces of Lipschitz functions.
\begin{eks}
Set
\begin{align*}
    X&=\{x_i\colon i\in\N\},\\
    Y&=\{y_i\colon i\in\N\},\\
    U&=\{u_i\colon i\in\N\},\\
    V&=\{v_i\colon i\in\N\},
\end{align*}
and $M=X\cup Y\cup U\cup V$. Define distance by
\[d(p,q)=\begin{cases} 
    1, & p=u_i, q\in \{x_j,u_j\} \text{ for some }i> j,\\  
    1, & q=u_i, p\in \{x_j,u_j\} \text{ for some }i> j,\\  
    1, & p=v_i, q\in \{y_j,v_j\} \text{ for some }i> j,\\ 
    1, & q=v_i, p\in \{y_j,v_j\} \text{ for some }i> j,\\ 
    2, & \text{otherwise}.\end{cases}\]
Define $f\colon M\rightarrow\mathbb{R}$ by
\[f(p)=\begin{cases}
        2, & p\in X\cup U,\\
        0, & p\in Y\cup V.
        \end{cases}\]
        
Note that $f(m_{pq})>0$ if and only if $p\in X\cup U$ and $q\in Y\cup V$ and in that case $f(m_{pq})=1$.

First we show that $f$ is a $w^*$-Daugavet-point. Fix $\alpha>0$ and $\mu\in S_{\mathcal{F}(M)}$ with finite support. We know that elements with finite support are dense in the unit sphere and therefore it suffices to show that there exists $h\in S(\mu,\alpha)$ such that $\|f-h\|=2$. Let $g\in S(\mu,\alpha)$ and let $n\in \N$ be such that $\supp(\mu)\subseteq \cup_{i=1}^n\{x_i,y_i,u_i,v_i\}$. Set
\[a=\frac{1}{2}\Big(\max \big\{g(p)\colon p\in\cup_{i=1}^n\{x_i,y_i,u_i,v_i\}\big\}+\min \big\{g(p)\colon p\in\cup_{i=1}^n\{x_i,y_i,u_i,v_i\}\big\}\Big)\]
and define $h\colon M\rightarrow\mathbb{R}$ by
\[h(p)=\begin{cases}
        g(p), & p\in \cup_{i=1}^n\{x_i,y_i,u_i,v_i\},\\
        a-1, & p\in \cup_{i=n+1}^\infty \{x_i\},\\
        a+1, & p\in \cup_{i=n+1}^\infty \{y_i\},\\
        a, & p\in \cup_{i=n+1}^\infty \{u_i,v_i\}.
        \end{cases}\]
Then $\|h\|\le 1$ because 
\begin{align*}
    |a-g(p)|&\le\frac{1}{2}\Big(\max \big\{g(p)\colon p\in\cup_{i=1}^n\{x_i,y_i,u_i,v_i\}\big\}\\
    &\qquad-\min \big\{g(p)\colon p\in\cup_{i=1}^n\{x_i,y_i,u_i,v_i\}\big\}\Big)\le 1
\end{align*}
for every $p\in \cup_{i=1}^n\{x_i,y_i,u_i,v_i\}$.
Clearly, $h(\mu)=g(\mu)>1-\alpha$ and 
\[\|f-h\|\ge f(m_{x_{n+1}y_{n+1}})-h(m_{x_{n+1}y_{n+1}})=2.\]
Therefore $f$ is a $w^*$-Daugavet-point.

Finally, we show that $f$ is not a $\Delta$-point. For $\varepsilon\in(0,1/2)$ and $m\in \N$, let us have functions $g_1,\ldots,g_k\in \Delta_\varepsilon(f)$. Then for each $i\in \{1,\ldots,k\}$ there exists $p_i,q_i\in M$ such that $f(m_{p_iq_i})-g_i(m_{p_iq_i})>2-2\varepsilon$. Then $f(m_{p_iq_i})>0$ and therefore $p_i\in X\cup U$ and $q_i\in Y\cup V$ for every $i\in \{1,\ldots,k\}$. Let $n\in\N$ be such that $\{p_1,\ldots,p_k,q_1,\ldots,q_k\}\subseteq \cup_{i=1}^{n}\{x_i,y_i,u_i,v_i\}$. By {\cite[Lemma~1.2]{Vee}} we have
\begin{align*}
    \|m_{u_{n+1}v_{n+1}}-m_{p_iq_i}\|&= \frac{d(u_{n+1},p_i)+d(q_i,v_{n+1})+|d(u_{n+1},v_{n+1})-d(p_i,q_i)|}{\max\big\{d(u_{n+1},v_{n+1}),d(p_i,q_i)\big\}}\\
    &=\frac{1+1+|2-2|}{\max\big\{2,2\big\}}\\
    &=1
\end{align*}
for every $i\in \{1,\ldots,k\}$. Therefore
\[g_i(m_{u_{n+1}v_{n+1}})<g_i(m_{u_{n+1}v_{n+1}})-g_i(m_{p_iq_i})-1+2\varepsilon\le 2\varepsilon\]
for every $i\in \{1,\ldots,k\}$. Hence \[f(m_{u_{n+1}v_{n+1}})-\sum_{i=1}^k\lambda_ig_i(m_{u_{n+1}v_{n+1}})>1-2\varepsilon\]
for all $\lambda_1,\ldots,\lambda_k\ge0$ with $\sum_{i=1}^k\lambda_i=1$. Thus $f\notin \clconv\Delta_\varepsilon(f)$, i.e., $f$ is not a $\Delta$-point.
\end{eks}

\section*{Acknowledgements}
This paper is a part of the author's Ph.D.\ thesis, which is being prepared at the University of Tartu under the supervision of Rainis Haller and Vegard Lima. 
The author is grateful to her supervisors for their valuable help and guidance. The author is also thankful to   Rainis Haller, Andre Ostrak, and Märt Põldvere for sharing their preprint.
 
This work was supported by the Estonian Research
Council grant (PRG1901).


\footnotesize
\begin{thebibliography}{99}

\bibitem{AHLP} Trond A. Abrahamsen, Rainis Haller, Vegard Lima, and Katriin Pirk, \emph{Delta- and Daugavet-points in Banach spaces}, Proc. Edinb. Math. Soc. (2) \textbf{63} (2020), no. 2, 475--496.

\bibitem{ALM} Trond A. Abrahamsen, Vegard Lima, Andr\'{e} Martiny, and Stanimir Troyanski, \emph{Daugavet- and delta-points in {B}anach spaces with unconditional bases}, Trans. Amer. Math. Soc. Ser. B \textbf{8} (2021), 379--398.

\bibitem{ALMP} Trond A. Abrahamsen, Vegard Lima, Andr\'{e} Martiny,  and Yoël Perreau, \emph{Asymptotic geometry and delta-points}, arXiv:2203.14528 [math.FA] (2022).

\bibitem{CGMR} Rafael Chiclana, Luis Garc\'{\i}a-Lirola, Miguel Mart\'{\i}n, and Abraham Rueda Zoca, \emph{Examples and applications of the density of strongly norm attaining {L}ipschitz maps}, Rev. Mat. Iberoam. \textbf{37} (2021), no. 5, 1917--1951.


\bibitem{CJ} Marek C\'{u}th, and Michal Johanis, \emph{Isometric embedding of {$\ell_1$} into {L}ipschitz-free spaces and {$\ell_\infty$} into their duals}, Proc. Amer. Math. Soc. \textbf{145} (2017), no. 8, 3409--3421.


\bibitem{DJR} Sheldon Dantas, Mingu Jung, and Abraham Rueda Zoca, \emph{Daugavet points in projective tensor products}, 	Q. J. Math. (2021), haab036.

\bibitem{GPR} Luis Garc\'{\i}a-Lirola,  Anton\'{\i}n Proch\'{a}zka, and Abraham Rueda Zoca, \emph{A characterisation of the {D}augavet property in spaces of {L}ipschitz functions}, J. Math. Anal. Appl. \textbf{464} (2018), no. 1, 473--492.

\bibitem{GOD} Gilles Godefroy, \emph{A survey on {L}ipschitz-free {B}anach spaces}, Comment. Math. \textbf{55} (2015), no. 2, 89--118.

\bibitem{HOP} Rainis Haller, Andre Ostrak, and Märt Põldvere, \emph{Diameter two properties for spaces of Lipschitz functions}, 	arXiv:2205.13287 [math.FA] (2022).

\bibitem{HPV} Rainis Haller, Katriin Pirk, and Triinu Veeorg, \emph{Daugavet- and delta-points in absolute sums of {B}anach spaces}, J. Convex Anal. \textbf{28} (2021), no. 1, 41--54.

\bibitem{IK}
Yevgen Ivakhno and Vladimir Kadets, \emph{Unconditional sums of spaces with bad projections}, Visn. Khark. Univ., Ser. Mat. Prykl. Mat. Mekh. \textbf{645} (2004), no.~54, 30--35.

\bibitem{JRZ} Mingu Jung and Abraham Rueda Zoca, \emph{Daugavet points and {$\Delta$}-points in {L}ipschitz-free spaces}, Studia Math. \textbf{265} (2022), no.~1, 37--55.

\bibitem{Vee} Triinu Veeorg, \emph{Characterizations of Daugavet- and delta-points in Lipschitz-free spaces}, Studia Math.,  to appear, arXiv:2111.14393 [math.FA] (2021).

\bibitem{Weaver} Nik Weaver, \emph{Lipschitz algebras}, World Scientific Publishing Co. Pte. Ltd., Hackensack, NJ, 2018, Second edition.

\bibitem{DW01} Dirk~Werner, \emph{Recent progress on the Daugavet property}, Irish Math. Soc. Bull. \textbf{46} (2001), 77--97.


\end{thebibliography}

\bibliographystyle{amsplain}

\end{document}